\documentclass[a4paper, 11pt]{amsart}
\usepackage{amsmath}
\usepackage{amsthm}
\usepackage{amssymb}
\usepackage{xypic}
\usepackage[only,mapsfrom]{stmaryrd}
\usepackage{graphicx}
\usepackage{color}

\newtheorem{theorem}{Theorem}[section]
\newtheorem{lemma}[theorem]{Lemma}

\newtheorem{prop}[theorem]{Proposition}
\newtheorem*{thmint}{Main Theorem}
\newtheorem{definition}[theorem]{Definition}

\theoremstyle{remark}
\newtheorem{remark}[theorem]{Remark}

\numberwithin{equation}{section}

\newcommand{\C}{\mathbb{C}}

\DeclareMathOperator{\image}{im}

\DeclareMathOperator{\codim}{codim}

\DeclareMathOperator{\id}{Id}
\DeclareMathOperator{\Id}{Id}
\DeclareMathOperator{\Homeo}{Homeo}

\title{Positively curved $10$-manifolds with $T^3$-symmetry}

\author{Anusha M. Krishnan}
\address{
Mathematisches Institut\\ Universit\"at M\"unster\\Einsteinstrasse 62\\D-48149 M\"unster\\Germany
}
\curraddr{
Department of Mathematics\\
Indian Institute of Technology Bombay\\
Mumbai -- 400076\\
India
}
\email{anushamk@math.iitb.ac.in}

\author{Michael Wiemeler}
\address{
Mathematisches Institut\\ Universit\"at M\"unster\\Einsteinstrasse 62\\D-48149 M\"unster\\Germany
}
\email{wiemelerm@uni-muenster.de}


\thanks{The research for this paper was funded by the Deutsche Forschungsgemeinschaft (DFG, German Research Foundation) under Germany's Excellence Strategy EXC 2044 --390685587, Mathematics M\"unster: Dynamics--Geometry--Structure and through CRC1442 Geometry: Deformations and Rigidity at University of M\"unster.}
\date{\today}

\begin{document}

\begin{abstract}
  We show that closed, simply connected, positively curved 10-manifolds with effective, isometric actions of $3$-dimensional tori are homotopy spheres or homotopy complex projective spaces.
\end{abstract}

\maketitle

\section{Introduction}
\label{sec:intro}

The classification of simply connected closed manifolds which admit metrics of positive sectional curvature is a long-standing open problem in Riemannian geometry.  In addition to the paucity of examples, all known obstructions to positive sectional curvature are also obstructions to non-negative sectional curvature. Nonetheless, major breakthroughs have been made in understanding positively curved manifolds in the presence of symmetry \cite{MR992332,MR1255926,MR2051400,MR2139252,zbMATH06152271,KWW,Nienhaus,KWW2}.

One measure of the amount of symmetry of a Riemannian manifold, is its {\it symmetry rank}, defined to be the maximal dimension of a torus that acts isometrically and effectively on the manifold.  By \cite{MR1255926}, it is known that the symmetry rank of a positively curved closed \(2n\)-dimensional manifold \(M\) is bounded above by \(n\), and in case of equality \(M\) is diffeomorphic to \(S^{2n}, \mathbb{C}P^n\) or \(\mathbb{R} P^{2n}\).  However, there are classification results (up to homeomorphism or diffeomorphism) with lesser symmetry rank assumptions in small dimensions, for example the Gauss--Bonnet theorem in dimension $2$, \cite{MR992332} in dimension $4$ and \cite{MR2139252} in dimension $8$.
These results characterize positively curved manifolds of symmetry rank at least \(0\), \(1\) and \(3\), respectively.

Moreover, it follows from \cite{Nienhaus} that the Euler characteristic of \(M\) is positive if \(M\) is a positively curved, even-dimensional manifold of symmetry rank at least \(4\).



From the main result of \cite{MR2051400}, one sees that a closed, simply connected, positively curved $10$-manifold with symmetry rank $4$ has the homotopy type of $S^{10}$ or of $\C P^5$.  In this paper we show that the same conclusion holds under the weaker assumption of symmetry rank $3$.

\begin{thmint}
  \label{thm:main}
  Let \(M\) be a closed, simply connected, positively curved \(10\)-manifold with an isometric and effective $T^3$-action.
  Then \(H^*(M;\mathbb{Z})\) is isomorphic as a graded ring to \(H^*(S^{10};\mathbb{Z})\) or to \(H^*(\mathbb{C}P^5;\mathbb{Z})\).
  In particular, \(M\) is a homotopy sphere or a homotopy complex projective space.
\end{thmint}

Note that if we are in the second case of the theorem then work of Dessai and Wilking \cite{zbMATH02133491} on Petrie's conjecture implies that the Pontrjagin classes of \(M\) are standard, i.e. are the same as for a standard \(\mathbb{C} P^5\).
Combining this with simply connected surgery theory implies that the number of diffeomorphism types of manifolds as in the Main Theorem is finite \cite[p.51-52]{novikov10:_homot}.

It should be mentioned that \(S^{10}\) and \(\mathbb{C}P^5\) are the only known examples of closed, simply connected 10-manifolds which admit metrics of positive curvature.
\smallskip

Next we will describe the proof strategy.  The main ingredients are results of Amann and Kennard \cite{MR4130255}, localization theorems in equivariant cohomology (see \cite{zbMATH03988265}, \cite{zbMATH00051915}) and GKM theory (see \cite{MR1489894}, \cite{MR3456711}).

To be more precise, in \cite{MR4130255}, the authors used the connectedness lemma of \cite{MR2051400} together with some representation theory to obtain partial results for the cohomology ring  of a closed, simply connected, positively curved \(10\)-manifold \(M\) with \(T^3\)-symmetry.
In all except one case they determine the possible cohomology rings completely.
Here, we analyze the exceptional case and show that in this case \(M\) has the integral cohomology ring of \(\mathbb{C} P^5\).

Since the codimension of the fixed set of a circle action on a manifold is even, and the case of fixed point sets of codimension two is understood by \cite{MR1255926},
we have to study two subcases:
\begin{enumerate}
\item There is an \(S^1\subset T^3\) with a $S^1$-fixed point component of dimension $6$.
\item For all \(S^1\subset T^3\), the $S^1$-fixed point set is at most $4$-dimensional.
\end{enumerate}

The first case follows by results from \cite{MR2139252} and \cite{KWW}.
The second case is more involved.
In this case the action is of type GKM$_3$.
Using a spectral sequence argument we show that the rational cohomology of \(M\) is concentrated in even degrees.
When this is done, the computation of the rational cohomology ring of \(M\) follows from results in \cite{MR3456711}. 
Then it remains to show that the integral cohomology of \(M\) is torsion-free.
Our argument to show this is based on localization theorems in torus equivariant cohomology \cite{zbMATH03988265,zbMATH00051915} and classifications of positively curved manifolds with torus symmetry in dimensions less than $10$ \cite{MR992332,MR1255926,MR2139252}.

The remaining sections of this paper are structured as follows.
In Section \ref{sec:prelim} we collect the necessary preliminaries.
In Sections~\ref{sec:dim6} and~\ref{sec:dim4} we discuss the above mentioned subcases (1) and (2), respectively.
Finally, in Section~\ref{sec:complete} we complete the proof of the main theorem.

\subsection*{Acknowledgements}
 We would like to thank the anonymous referee for detailed comments which helped to improve the presentation of this paper.
We thank Burkhard Wilking for suggesting to look at torus actions on low dimensional positively curved manifolds.

\section{Preliminaries}
\label{sec:prelim}

In this section we collect some known results on positively curved manifolds with symmetry and also on the algebraic topology of smooth torus actions.

\subsection{Positive curvature and symmetry}
\label{sec:prelim_pos_sec}

Here we collect some known results on torus actions on positively curved manifolds.

The first result which we need is the following theorem due to Hsiang and Kleiner.
It gives a classification of positively curved manifolds with circle action in dimension four.

\begin{theorem}[\cite{MR992332}]
  \label{thm:hsiang-kleiner}
  Let \(M^4\) be a closed, simply connected, positively curved manifold with an effective, isometric \(S^1\)-action.
  Then \(M\) is homeomorphic to \(S^4\) or \(\mathbb{C} P^2\).
\end{theorem}

It has been shown by Grove and Searle that the dimension of a torus which acts effectively and isometrically on an \(n\)-dimensional positively curved manifold is bounded from above by \([\frac{n+1}{2}]\).
For the case that the dimension of the acting torus is maximal they also showed:

\begin{theorem}[\cite{MR1255926}]
  \label{thm:grove_searle}
  Let \(M^n\) be a closed, simply connected, positively curved manifold. Assume that one of the following two conditions is satisfied.
  \begin{enumerate}
  \item There is an effective, isometric \(S^1\)-action on \(M\) with \(\codim M^{S^1}=2\).
  \item $M$ has symmetry rank at least $n/2$.
  \end{enumerate}
  Then \(M\) is diffeomorphic to \(S^{n}\) or \(\mathbb{C} P^{n/2}\).
\end{theorem}

This diffeomorphism classification has been extended to a homeomorphism classification by Fang and Rong in the case that the manifold is of dimension $8$ and the symmetry rank is almost maximal.

\begin{theorem}[{\cite[Theorem A]{MR2139252}}]
  \label{thm:fang_rong1}
  Let \(M^8\) be a closed, simply connected, positively curved manifold with an effective, isometric \(T^3\)-action.
  Then \(M\) is homeomorphic to \(S^8\), \(\mathbb{C} P^4\) or \(\mathbb{H} P^2\).
\end{theorem}

From their proof we also need the following lemma.

\begin{lemma}[{\cite[Lemma 2.1]{MR2139252}}]
  \label{lem:fang_rong2}
  Let \(M^n\) be a closed, simply connected, positively curved manifold. Assume that \(M\) admits an effective, isometric \(T^{ [(n-1)/2 ]}\) -action. If there is a non-trivial isotropy group with codimension 2 fixed point set, then \(M\) is homeomorphic to a sphere or a complex projective space.
\end{lemma}

We also need the following topological result in dimension seven.

\begin{prop}[{\cite[Proposition 8.1]{MR2139252}}]
  \label{prop:fang_rong3}
  Let \(M^7\) be a closed, simply connected manifold. Assume that \(M\) admits an effective \(T^3\)-action such that
\begin{enumerate}
\item  there is no circle subgroup with fixed point set of codimension 2;
\item  all isotropy groups are connected;
\item  each fixed point component of any isotropy group is either a circle or a
lens space of dimension 3.
\end{enumerate}
Then \(M\) is homeomorphic to \(S^7\) if and only if \(b_2 (M) = 0\).
\end{prop}

All of the above results have been generalized by Wilking as follows.

\begin{theorem}[{\cite[Theorem 2]{MR2051400}}]
\label{thm:wilking1}
  Let \(M^n\) be a closed, simply connected, positively curved manifold, \(n\geq 10\), and let $d\geq \frac{n}{4}+1$.
Suppose that there is an effective, isometric action of a torus \(T^d\) on \(M\).
Then \(M\) is homeomorphic to \(\mathbb{H} P^{n/4}\) or to \(S^n\), or \(M\) is homotopy equivalent to \(\mathbb{C} P^{n/2}\).
\end{theorem}

As a major ingredient in the proof of this result he showed:

\begin{theorem}[{\cite[Theorem 2.1]{MR2051400}}]
  \label{thm:wilking2}
  Let \(M^n\) be a compact Riemannian manifold with positive sectional curvature.
  \begin{enumerate}
  \item  Suppose that \(N^{n -k} \subset M^n\) is a compact totally geodesic embedded submanifold of
codimension \(k\). Then the inclusion map \(N\hookrightarrow M\)  is \((n - 2 k + 1 )\)-connected. If there is
a Lie group \(G\) acting isometrically on \(M\) and fixing \(N\) pointwise, then the inclusion
map is \(( n - 2 k + 1 +\delta(G) )\)-connected, where \(\delta(G)\) is the dimension of a principal orbit.
\item Suppose that \(N_1^{n-k_1}, N_2^{n-k_2} \subset M^n\) are two compact totally geodesic embedded submanifolds, \(k_1 \leq k_2\), \(k_1 +k_2 \leq n \). Then the intersection \(N_1\cap N_2\) is a totally geodesic embedded submanifold as well, and the inclusion
   \[N_1\cap N_2\hookrightarrow N_2\]
is \(( n - k_1 - k_2)\)-connected.
  \end{enumerate}
\end{theorem}

The following result shows that the above theorem has strong implications on the topology of the manifold.

\begin{lemma}[{\cite[Lemma 2.2]{MR2051400}}]
  \label{lem:wilking3}
  Let \(M^n\) be a closed, differentiable, oriented manifold, and let \(N^{n-k}\) be an embedded, compact, oriented submanifold without boundary. Suppose that the inclusion
\(\iota: N \hookrightarrow M\)  is \(( n - k - l )\)-connected and \(n - k - 2 l > 0\) .
Let \([N] \in H_{ n - k} ( M , \mathbb{Z})\) be the image of the fundamental class of \(N\) in \(H_* ( M , \mathbb{Z})\), and let \(e \in  H^k ( M , \mathbb{Z})\) be its Poincar\'e
dual. Then the homomorphism
\[\cup e: H^i ( M , \mathbb{Z}) \rightarrow H^{i+k}(M, \mathbb{Z})\]
is surjective for \(l \leq i < n - k - l\) and injective for \(l < i \leq n - k - l\).
\end{lemma}

By Theorem~\ref{thm:wilking1}, a closed, simply connected, positively curved 10-manifold with \(T^4\)-symmetry is a homotopy $S^{10}$ or a homotopy $\C P^5$.
For such manifolds with isometric \(T^3\)-action \cite{MR4130255} contains a partial classification result.
We reformulate their results as follows.

\begin{theorem}
\label{thm:amann_kennard}
Let \(M^{10}\) be a closed, simply connected, positively curved manifold with \(T^3\)-symmetry.
Then one of the following three statements holds:
\begin{enumerate}
\item There is an involution in \(T^3\) fixing a $8$-dimensional submanifold of \(M\). In this case \(M\) is homeomorphic to \(S^{10}\) or \(\mathbb{C} P^5\).
\item \(M\) is a homology sphere, i.e., \(H^*(M;\mathbb{Z})=H^*(S^{10};\mathbb{Z})\).
\item All of the following statements hold: 
  \begin{itemize}
\item \(H_i(M;\mathbb{Z})=H_i(\mathbb{C}P^5;\mathbb{Z})\) for \(i\leq 3\)
\item \(H^4(M;\mathbb{Q})\neq 0\)
\item For the Euler characteristic of \(M\) we have \(\chi(M)=\chi(\mathbb{C}P^5).\)
\item There are two involutions in \(T^3\) fixing $6$-dimensional submanifolds \(N_1,N_2\subset M\).
\end{itemize}
\end{enumerate}
Here, by an involution we mean an element of order two in \(T^3\).
\end{theorem}

\begin{proof}
  Case (1) follows from \cite[Lemma 4.1]{MR4130255}. If we are not in case (1) then by \cite[Theorem 5.1 and Lemma 5.2]{MR4130255}, we are in case (2) or (3).
\end{proof}

We also need the following version of \cite[Theorem 4.1]{KWW}.

\begin{theorem}\label{thm:kww} Suppose that the closed, simply connected, positively curved manifold $M^{4n+2}$ has the rational cohomology of $S^2\times \mathbb{H}P^n$.
We assume $M$ is endowed with an effective, isometric $T^2$-action with a fixed-point component
$F_0$, such that \(F_0\) is the transverse intersection of two submanifolds \(N_1\) and \(N_2\) which are fixed by circle subgroups of \(T^2\). 
Assume that \(N_i\)
has the rational cohomology of $S^{2}\times \mathbb{H}P^{n-k_i}$ for $i=1,2$.

Then there is a second \(T^2\)-fixed-point component $F_1$ with $H^*(F_1;\mathbb{Q})\cong H^*(S^2\times \mathbb{C}P^l;\mathbb{Q})$ 
with $l>0$.
\end{theorem}
\begin{proof}
  We may assume that \(k_2\geq k_1\).
  Then, by Theorem~\ref{thm:wilking2}, \(F_0\hookrightarrow N_2\) is \((2+4(n-k_1-k_2))\)-connected.
  In particular, \(F_0\) has the rational cohomology of \(S^2\times \mathbb{H} P^{n-k_1-k_2}\).
  
  So the claim of the theorem follows from a combination of Theorem 4.1 in \cite{KWW} and the proof of the main result in \cite{MR1255926}. Indeed, the option b) of the theorem in \cite{KWW} cannot occur because in the presence of positive curvature an isometric circle action can have at most one codimension-two fixed point component.
\end{proof}

\subsection{Equivariant cohomology and localization}
\label{sec:prelim_equi_cohom}

In this section we recall the basic results on equivariant cohomology and localization.
For more details on this subject we refer the reader to \cite[Chapters 5 and 6]{zbMATH00051915} and \cite[Chapter 3]{zbMATH03988265}.

Let \(T^d\) be a \(d\)-dimensional torus.
For convenience we restrict our discussion to \(T^d\)-spaces which are finite \(T^d\)-CW-complexes (see \cite[p. 98]{zbMATH03988265} for a precise definition).
Using equivariant Morse theory one can show that every compact smooth \(T^d\)-manifold is equivariantly homotopy equivalent to a finite \(T^d\)-CW-complex \cite[Satz 3.3]{zbMATH04100211}.
Moreover, if \(H\subset T^d\) is a subgroup and \(X\) is a finite \(T^d\)-CW-complex, then the orbit space \(X/H\) and the \(H\)-fixed point sets are finite \((T^d/H)\)-CW-complexes.

In the following let \(R=\mathbb{Q}\) or \(R=\mathbb{Z}\) and \(X\) a finite \(T^d\)-CW-complex.
The equivariant cohomology of \(X\) with coefficients in \(R\) is defined as
\[H^*_{T^d}(X;R)=H^*(X_{T^d};R),\]
where \(X_{T^d}=ET^d\times_{T^d}X\) is the Borel construction of \(X\), \(ET^d\rightarrow BT^d\) is the universal principal \(T^d\)-bundle and \(X\rightarrow X_{T^d}\rightarrow BT^d\) is the Borel fibration.

Then the projection \(\pi:X_{T^d}\rightarrow BT^d\) induces an \(H^*(BT^d;R)\)-algebra structure on \(H^*_{T^d}(X;R)\).
Note that \(H^*(BT^d;R)=R[t_1,\dots,t_d]\) with \(\deg t_i=2\) for all \(i=1,\dots,d\).

We say that the \(T^d\)-action on \(X\) is almost free if all isotropy groups of points in \(X\) are finite.
In this case the equivariant cohomology of \(X\) can be computed as follows.

\begin{lemma}
  \label{sec:equiv-cohom-local-4}
  Let \(X\) be a finite almost free \(T^d\)-CW-complex.
  Then the projection \(X_{T^d}\rightarrow X/T^d\) induces an isomorphism
  \[H^*_{T^d}(X;\mathbb{Q})\cong H^*(X/T^d;\mathbb{Q}).\]
\end{lemma}
\begin{proof}

  We prove this by induction on the number of orbit types in \(X\).
  First assume that there is only one orbit type with isotropy group \(H\).
  Then the map \(X_{T^d}\rightarrow X/T^d\) is a fibration with fiber \(BH\).
  Since \(H\) is finite \(BH\) is rationally acyclic. Hence the claim follows in this case.

  Next assume that \(X\) has \(k\) different orbit types and the claim is proven for all \(T^d\)-CW complexes with at most \(k-1\) orbit types.
  Note that the isotropy subgroups of \(T^d\) are partially ordered by inclusion.
  Let \(H\subset T^d\) be an isotropy subgroup which is maximal with respect to this partial ordering and \(U\subset X\) an open invariant neighborhood of \(X^H\) which is equivariantly homotopy equivalent to \(X^H\).
  Then \(X-X^H\) and \((X-X^H)\cap U\) have at most \((k-1)\) orbit types and \(U\) is equivariantly homotopy equivalent to a \(T^d\)-CW-complex with one orbit type.
  Hence the claim follows from the 5-lemma and an inspection of Mayer-Vietoris sequences.

  This proves the lemma.
\end{proof}

There are several spectral sequences related to the equivariant cohomology of \(X\).
First there is the Serre spectral sequence  for the Borel fibration \(X\rightarrow X_{T^d}\rightarrow BT^d\)  (see \cite[Theorem 5.2]{MR1793722}).
Since the structure group of the bundle \(X_{T^d}\rightarrow BT^d\) is connected, the universal coefficient theorem implies that the \(E_2\)-page of this spectral sequence is given by
\[H^*(X;R)\otimes_RH^*(BT^d;R).\]
Moreover, the spectral sequence converges to \(H^*_{T^d}(X;R)\).
If \(H^*(X;R)\) is concentrated in even degrees this spectral sequence degenerates at the \(E_2\)-page. Hence, an application of \cite[Proposition 3.1.18]{zbMATH03988265} yields the following.

\begin{lemma}
\label{sec:equiv-cohom-local}
  Let \(X\) be a finite \(T^d\)-CW-complex.
  Assume that \(H^*(X;R)\) is a free \(R\)-module concentrated in even degrees.
  Then \(H^*_{T^d}(X;R)\) is  a free \(H^*(BT^d;R)\)-module concentrated in even degrees.
\end{lemma}

\begin{definition}
If $X$ is an orientable, closed, connected $n$-manifold, or more generally, if there is an $n \geq 0$ such that $H^i(X;R)=0$ for $i>n$ and $H^n(X;R)=R$, then, from the above Serre spectral sequence $E^{*, *}_*$, we get a homomorphism
\begin{align*}
  \pi_!:H^*_{T^d}(X,R)\rightarrow E_\infty^{n,*-n}\hookrightarrow E_2^{n,*-n}&=H^n(X;R)\otimes_RH^{*-n}(BT^d;R)\\ &\cong H^{*-n}(BT^d;R).
\end{align*}

The map \(\pi_!\) is called the {\bf integration over the fiber} (see \cite[Section 8]{zbMATH03159124}).
If \(X\) is a \(T^d\)-manifold we also write \(x[X]\) for \(\pi_!(x)\), \(x\in H^*_{T^d}(X;R)\).
\end{definition}

 The structure group of the \(T^d\)-bundle \(ET^d\times X\rightarrow X_{T^d}\) is connected.
Hence, the Serre spectral sequence for the fibration \(T^d\rightarrow ET^d\times X\rightarrow X_{T^d}\) has \(E_2\)-page
\[H^*_{T^d}(X;R)\otimes_RH^*(T^d;R)\]
and converges to \(H^*(X;R)\cong H^*(X\times ET^d;R)\).
If the \(T^d\)-action on \(X\) is almost free we therefore get from Lemma \ref{sec:equiv-cohom-local-4} the following lemma.

\begin{lemma}
\label{sec:equiv-cohom-local-1}
  Let \(X\) be a finite almost free \(T^d\)-CW-complex.
  Then there is a spectral sequence with \(E_2\)-page
  \[H^*(X/T^d;\mathbb{Q})\otimes_{\mathbb{Q}}H^*(T^d;\mathbb{Q})\]
  which converges to \(H^*(X;\mathbb{Q})\).
\end{lemma}

We now turn to localization techniques in equivariant cohomology.
We first have to fix some notation.

\begin{definition}
  Let \(S\) be a multiplicatively closed subset of homogeneous elements in \( H^*(BT^d;R)\). We then define
  \(\mathfrak{H}(S)\) to be the set of those closed subgroups \(H\) of \(T^d\) which satisfy
  \[S\cap \ker(H^*(BT^d;R)\rightarrow H^*(BH;R))=\emptyset.\]

  Moreover, we define \[X_S=\{x\in X;\;T_x^d\in \mathfrak{H}(S)\}.\]
Here \(T_x^d\subset T^d\) denotes the isotropy subgroup of \(x\in X\).
\end{definition}

With this notation we have the localization theorem.

\begin{theorem}[{\cite[Theorem 3.3.8]{zbMATH03988265}}]
\label{sec:equiv-cohom-local-2}
  Let \(X\) be a finite \(T^d\)-CW-complex and \(S\) a multiplicatively closed subset of  homogeneous elements in \(H^*(BT^d;R)\).
  Then the inclusion \(X_S\rightarrow X\) induces an isomorphism
  \[S^{-1}H^*_{T^d}(X;R)\cong S^{-1}H^*_{T^d}(X_S;R).\]
  Here for an \(H^*(BT^d)\)-module \(M\), \(S^{-1}M\) denotes the localization of \(M\) with respect to \(S\).
\end{theorem}

To apply this theorem we need to know what \(X_S\) is for special choices of \(S\).
Such knowledge is provided by the following lemma.

\begin{lemma}
  \label{sec:equiv-cohom-local-3}
  \begin{enumerate}
  \item Let \[S_0=\{\prod_i s_i\in H^*(BT^d;R);\; s_i\in H^2(BT^d;R)\setminus \{0\}\}.\]
    Then \(\mathfrak{H}(S_0)=\{T^d\}\). In particular, \(X_{S_0}=X^{T^d}\).
  \item  Let \(p\in \mathbb{Z}\) be a prime and \[S_p=\{\prod_i s_i\in H^*(BT^d;\mathbb{Z});\; s_i\in H^2(BT^d;\mathbb{Z})\setminus \{0\}, p \not| s_i\}.\]
    Then \(\mathfrak{H}(S_p)=\{H\subset T^d;\mathbb{Z}_p^d \subset H\}\), where \(\mathbb{Z}_p^d\subset T^d\) is the subgroup of all elements of order \(p\). In particular, \(X_{S_p}=X^{\mathbb{Z}_p^d}\).
  \end{enumerate}
\end{lemma}
\begin{proof}
  We first prove (1). Let \(H\subset T^d\) be a proper closed subgroup.
  Then there is a splitting \(T^d=T^{d-1}\times S^1\) such that \(H\subset K=T^{d-1}\times \mathbb{Z}_m\) with some cyclic subgroup \(\mathbb{Z}_m\) of order \(m\) of \(S^1\).

  From the K\"unneth theorem we have
  \[H^2(BT^{d};R)\cong H^2(BT^{d-1};R)\oplus H^2(BS^1;R)=H^2(BT^{d-1};R)\oplus t\cdot R\]
  with \(t\in H^2(BS^1;R)\) a generator.
Similarly we have
\[H^2(BK;R)\cong H^2(BT^{d-1};R)\oplus H^2(B\mathbb{Z}_m;R)=H^2(BT^{d-1};R)\oplus t_m\cdot R/mR,\]
where we write \(t_m\) for the image of \(t\) in \(H^2(B\mathbb{Z}_m;R)\).
Hence \(m\cdot t\in S_0\) restricts to zero in \(H^*(BK;R)\) and therefore also in \(H^*(BH;R)\).
Therefore \(H\) does not belong to \(\mathfrak{H}(S_0)\).
Since \(T^d\) clearly belongs to \(\mathfrak{H}(S_0)\), claim (1)  follows.

Next, we prove (2). Again, let \(H\) be a closed subgroup of \(T^d\).
If \(H\) is contained in some \(K=T^{d-1}\times \mathbb{Z}_m\) with \(p\not |m\),
then we can argue as in case (1) to see that \(H\) does not belong to \(\mathfrak{H}(S_p)\).

As a first step, we show that we can find such a \(K\) if \(H\) does not contain \(\mathbb{Z}_p^d\). So \(H\) does not belong to \(\mathfrak{H}(S_p)\) in this case.

As an abstract group, \(H\) is isomorphic to a product \(T^{d'}\times G_1\times G_2\), where
\begin{itemize}
\item \(T^{d'}\) is an \(d'\)-dimensional torus.
\item \(G_1=\prod_{i=1}^l \mathbb{Z}_{p^{k_i}}\) with some \(l\geq 0\), \(k_i\geq 1\).
\item \(G_2\) is a finite abelian group whose order is not a multiple of \(p\).
\end{itemize}

Note that \(T^{d'}\times G_1\) is contained in a subtorus of dimension at most \(d'+l\) in \(T^d\).
So if \(l<d-d'\) then \(T^{d'}\times G_1\subset T^{d-1}\) for some codimension one subtorus \(T^{d-1}\) of \(T^d\).
So in this case \(H\subset K=\langle T^{d-1},G_2\rangle\) where \(K\) is as above.
But if \(l\geq d-d'\) then \(H\) contains \(\mathbb{Z}_p^d\).

As a second step, we show that \(H\in \mathfrak{H}(S_p)\) if \(H\) contains \(\mathbb{Z}_p^d\).
  By definition, the set \(\mathfrak{H}(S_p)\) has the property that if \(K\subset K'\subset T^d\) are closed subgroups of \(T^d\) and \(K\in\mathfrak{H}(S_p)\) then \(K'\) is also in \(\mathfrak{H}(S_p)\).
So it only remains to show that \(\mathbb{Z}_p^d\in \mathfrak{H}(S_p)\).
For this note that we have isomorphisms
\[H^*(BT^d;\mathbb{Z}) \cong \mathbb{Z}[t_1,\dots,t_d]\]
and
\[H^2(B\mathbb{Z}_p^d;\mathbb{Z}) \cong t_{1p}\cdot \mathbb{Z}/p\mathbb{Z}\oplus\dots\oplus t_{dp}\cdot \mathbb{Z}/p\mathbb{Z}\]
where \(t_{ip}\) is the image of \(t_i\).
Moreover, from the Bockstein sequence one sees that there is an injection
\[H^2(B\mathbb{Z}_p^d;\mathbb{Z})\hookrightarrow H^2(B\mathbb{Z}_p^d;\mathbb{Z}/p\mathbb{Z})\subset  H^*(B\mathbb{Z}_p^d;\mathbb{Z}/p\mathbb{Z}).\]
If \(p>2\), then we have
\[ H^*(B\mathbb{Z}_p^d;\mathbb{Z}/p\mathbb{Z})=\mathbb{Z}/p\mathbb{Z}[t_{1p},\dots,t_{dp}]\otimes_{\mathbb{Z}/p\mathbb{Z}} \Lambda_p(s_1,\dots,s_d),\]
where \(\Lambda_p(s_1,\dots,s_d)\) denotes the free exterior algebra over \(\mathbb{Z}/p\mathbb{Z}\) on \(d\) generators of degree one.

If \(p=2\), then
\[ H^*(B\mathbb{Z}_p^d;\mathbb{Z}/p\mathbb{Z})=\mathbb{Z}/p\mathbb{Z}[s_1,\dots,s_d],\]
with \(s_i\) of degree one such that \(s_i^2=t_{ip}\) for all \(i=1,\dots,d\).

So if \(a_i=\sum_{j=1}^d\alpha_{ij} t_j\in H^2(BT^d;\mathbb{Z})\), \(\alpha_{ij}\in \mathbb{Z}\), \(i=1,\dots,k\), are such that \(\prod_{i=1}^k a_i\) restricts to zero in \(H^*(B\mathbb{Z}_p^d;\mathbb{Z})\), then this product also restricts to zero in \(H^*(B\mathbb{Z}_p^d;\mathbb{Z}_p)\).
Hence by the above remarks on the structure of this ring, it follows that there must be an \(i_0\in\{1,\dots,k\}\) such that \(a_{i_0}\) restricts to zero in \(H^2(B\mathbb{Z}_p^d;\mathbb{Z}_p)\).
But then each \(\alpha_{i_0j}\) must be divisible by \(p\).
So \(a_{i_0}\) does not belong to \(S_p\).
Since \(H^*(BT^d;\mathbb{Z})\) is a unique factorization domain, it follows that \(\prod_{i=1}^k a_i\not\in S_p\).

So \(\mathbb{Z}^d_p\in \mathfrak{H}(S_p)\) follows.
\end{proof}

\subsection{GKM manifolds}
\label{sec:gkm_manifolds}

We also need to recall some facts about GKM$_k$-manifold.
 GKM manifolds were first defined by Goresky, Kottwitz and MacPherson \cite{MR1489894}.
 The notion of a GKM$_k$ manifold was first defined in \cite{MR3456711}.
 Here we need a slightly more general definition.

\begin{definition}
  \label{sec:gkm-manifolds-6}
  Let \(M^{2n}\) be an orientable, closed manifold with \(T^d\)-action and \(2\leq k\leq d\).
We say that the torus action on \(M\) is GKM$_k$ if the following two conditions are satisfied:

  \begin{enumerate}
  \item For any subtorus \(T'\subset T^d\) and each component \(F\) of \(M^{T'}\) we have \(F\cap M^{T^d}\neq \emptyset\).
  \item For every point \(p\in M^{T^d}\) the following holds: Any \(k\) weights of the \(T^d\)-representation \(T_pM\) are linearly independent.
\end{enumerate}
\end{definition}

\begin{remark}
  \begin{enumerate}
  \item Note that if the action is GKM$_k$ for some \(k\geq 2\), then it is GKM$_{k'}$ for all \(2\leq k'\leq k\).
  \item Our definition of GKM$_k$ is more general than the usual definition (see for example \cite{MR3456711}) since it does not require that the action is equivariantly formal.
    \item If, in the situation of Definition~\ref{sec:gkm-manifolds-6}, there is a \(T\)-invariant metric of positive sectional curvature on \(M\), then it follows from a result of Berger \cite{MR133083} that our condition (1) is automatically satisfied.
  \end{enumerate}
\end{remark}

\begin{lemma}
\label{sec:gkm-manifolds-5}
  Let \(M^{2n}\) be orientable manifold with \(T^d\)-action.
  Assume that for any subtorus \(T'\subset T^d\) and each component \(F\) of \(M^{T'}\) we have \(F\cap M^{T^d}\neq \emptyset\).
  Then for \(2\leq k\leq d\) the following two statements are equivalent:
  \begin{enumerate}
  \item For every codimension-\((k-1)\) subtorus \(T'\) of \(T^d\), we have \(\dim M^{T'}\leq 2(k-1)\).
  \item For every point \(p\in M^{T^d}\) the following holds: Any \(k\) weights of the \(T^d\)-representation \(T_pM\) are linearly independent, i.e. the action is GKM$_k$.
  \end{enumerate}
\end{lemma}

\begin{proof}
  Let \(M\) and \(k\) be as in the lemma.

  Assume that (1) holds. To show that (2) holds we argue by contradiction.
  Let \(x\in M^{T^d}\) and \(\alpha_1,\dots,\alpha_n\in (LT^d)^*\) be the weights of the \(T^d\)-representation \(T_xM\).
  Here \((LT^d)^*\) denotes the dual Lie algebra of \(T^d\).
  Assume that \(\alpha_1,\dots,\alpha_k\) are linearly dependent.
  Then the identity component of \(\bigcap_{i=1}^k\ker\alpha_i\) contains a codimension \((k-1)\)-subtorus \(T'\subset T^d\).
  Since \(T_x(M^{T'})=(T_xM)^{T'}\) it follows that \(M^{T'}\) has dimension at least \(2k\), a contradiction to (1).

  Next assume that (2) holds.
  We want to show (1).
  Let \(T'\) be a subtorus of codimension \(k-1\).
  Let \(F\) be a component of \(M^{T'}\), \(x\in F^{T^d}\) and \(\alpha_1,\dots,\alpha_n\in (LT^d)^*\) the weights of the \(T^d\)-representation \(T_xM\).
  Let \(V_i\) be the irreducible \(T^d\)-representation of weight \(\alpha_i\).
  Then \(T_xM=\bigoplus_{i=1}^n V_i\) and \[T_xF=\bigoplus_{i;\alpha_i|_{T'}=0} V_i\]
  Since the kernel of the restriction map \((LT^d)^*\rightarrow (LT')^*\) has dimension \(k-1\) and any \(k\) of the \(\alpha_i\) are linearly independent, there are at most \(k-1\) of the \(\alpha_i\) which restrict to zero in \((LT')^*\).
  Because \(\dim V_i=2\) it follows that \(\dim F\leq 2(k-1)\).
  Since this holds for every component of \(M^{T'}\) statement (1) follows.
\end{proof}

One defines the one-skeleton of an action of \(T^d\) on a manifold \(M\) as
\[M_1=\{x\in M;\;\dim Tx\leq 1\}=\bigcup_{T'\subset T^d} M^{T'},\]
where the union is over all codimension-one subtori \(T'\)  in \(T^d\).
Note that if the action is GKM$_k$, \(k\geq 2\), then \(M_1\) is a union of copies of \(S^2\) acted on by \(T^d\) by rotations.
Hence \(M_1/T^d\) has the structure of a graph \(\Gamma_M\).
In this graph the vertices correspond to the fixed points of the action.
Any two of these vertices are connected by as many edges as there are \(T^d\)-invariant \(S^2\)'s containing both of the corresponding fixed points.

One can put labels on the edges of this graph given by the isotropy groups of an orbit in the interior of the edges.
This gives rise to a labeled graph \((\Gamma_M,\alpha)\).
This labeled graph is called the GKM graph of \(M\).

In presence of positive curvature the following has been shown by Goertsches and Wiemeler.

\begin{theorem}[{\cite[Section 5]{MR3456711}}]
\label{sec:gkm-manifolds-3}
  Let \(M\) be a positively curved GKM$_3$ manifold.
  Then each component of the labeled graph \((\Gamma_M,\alpha)\) is isomorphic to a graph obtained from a linear torus action on a CROSS of the same dimension as \(M\).
  Note that the CROSSes are given by \(S^n,\mathbb{C}P^n,\mathbb{H}P^n\) and \(\mathbb{O}P^2\).
\end{theorem}
\begin{proof}
    In \cite{MR3456711} the rational cohomology rings of equivariantly formal, positively curved GKM$_3$-manifolds were determined.
    
    The assumption of equivariant formality entered in the proof in two ways, to see that
    \begin{enumerate}
    \item the GKM graph of \(M\) is connected, and that
    \item the GKM graph of \(M\) determines the cohomology ring of \(M\) (see Theorem~\ref{sec:gkm-manifolds-4} below).
    \end{enumerate}

    The rest of the arguments in \cite{MR3456711} are geometric and combinatorial in nature.
    They were used to determine the GKM graph of a positively curved GKM$_3$-manifold.
    
    When we drop the assumption of equivariant formality, we lose the above two properties.
    But the arguments on the structure of the GKM graph from \cite[Section 5]{MR3456711} remain valid, if one argues for each of its components separately.
    So that essentially the above theorem is proved there. 
\end{proof}

The importance of the GKM graph \(\Gamma_M\) is shown by the following theorem due to Goresky, Kottwitz and MacPherson.

\begin{theorem}[\cite{MR1489894}]
  \label{sec:gkm-manifolds-4}
  Let \(M\) be a connected GKM$_k$ manifold with vanishing odd degree Betti numbers, where \(k\geq 2\).  Then \(\Gamma_M\) is connected and the rational ordinary and equivariant cohomology rings of \(M\) are determined by \((\Gamma_M,\alpha)\).
\end{theorem}

Similarly to the one-skeleton of a torus action on a manifold \(M\) one can define the two-skeleton of the action as
\[M_2=\{x\in M;\dim Tx \leq 2\}.\]
If the action on \(M\) is GKM$_3$ then \(M_2\) is a union of $4$-dimensional manifolds.

We need the following theorem.

\begin{theorem}
  \label{sec:gkm-manifolds-2}
  Let \(M\) and \(N\) be two GKM$_3$-manifolds of the same dimension such that all fixed point components of codimension-two subtori are simply connected.
  Then the two-skeleta of \(M\) and \(N\) are equivariantly homeomorphic if and only if their GKM graphs are isomorphic.
\end{theorem}

The proof of this theorem uses the following two lemmas.

\begin{lemma}
  \label{sec:gkm-manifolds}
  Let \(T^d\) act on the $2$-sphere \(M\) by rotations. Then every homeomorphism \(M^T\rightarrow M^T\) extends to an equivariant homeomorphism \(M\rightarrow M\).
\end{lemma}
\begin{proof}
  The orbit space of the action on \(M\) is a compact interval \([-1,1]\) with the end points corresponding to the fixed points of the action and interior points corresponding to principal orbits.
    There is a continuous section to the orbit map \(s:[-1,1]\rightarrow M\).
    This section induces an equivariant homeormophism
    \[([-1,1]\times T^d)/\!\!\sim\;\; \rightarrow M\quad\quad(x,t)\mapsto ts(x).\]
  
  Here, for   \((x,t), (x',t')\in [-1,1]\times T^d\) we have
  \((x,t)\sim (x',t')\) if and only if \(x=x'\) and
  \begin{itemize}
  \item \(x\in \partial [-1,1]\) or
  \item \(x\not\in \partial [-1,1]\) and \(t^{-1}t'\in T'\) where \(T'\subset T^d\) is the principal isotropy group of the action on \(M\).
  \end{itemize}
  
   Now, a homeomorphism \(f\) of \(M^T\) is nothing but a homeomorphism of \(\partial [-1,1]\).
   Each such homeomorphism extends to a homeomorphism \(F\) of \([-1,1]\).
   Now \(F\times \id_{T^d}\) induces an extension of \(f\) to an equivariant homeomorphism of \(M\).
\end{proof}

\begin{lemma}
  \label{sec:gkm-manifolds-1}
  Let \(M^4\) be a closed, simply connected manifold with an effective \(T^2\)-action.
  Then every equivariant homeomorphism \(M_1\rightarrow M_1\) extends to an equivariant homeomorphism \(M\rightarrow M\).
\end{lemma}
\begin{proof}
  Since \(M\) is simply connected and closed,
  we have that \(M/T^2\) is a manifold with corners homeomorphic to \(D^2\) by the classification results of Orlik and Raymond \cite[Sections 1.12, 5.1, 5.2]{zbMATH03343393}.

   Again, there is a continuous section \(s:M/T^2\rightarrow M\) to the orbit map.
  Hence, there is an equivariant homeomorphism
     \[(M/T^2\times T^2)/\!\!\sim\;\; \rightarrow M\quad\quad(x,t)\mapsto ts(x).\]
  Here we have \((x,t)\sim (x',t')\) if and only if \(x=x'\) and \(tt'^{-1}\) is contained in the isotropy group of the orbit \(x\).
  Note that these isotropy groups are constant on the interiors of the faces of \(M/T^2\).

  Moreover, \(M_1\) is given by the preimage of \(\partial (M/T^2)\) under the orbit map.
  Note that the equivariant homeomorphisms of \(S^2=([-1,1]\times T^2)/\sim\) are all of the form \(F(x,t)=(f(x),h(x)t)\) with some homeomorphism \(f:[-1,1]\rightarrow [-1,1]\) and a map \(h:[-1,1]\rightarrow T^2\) whose restriction to \(]-1,1[\) is continuous.
  
  Notice that all such homeomorphisms are isotopic to the homeomorphism induced by \((\pm\Id_{[-1,1]})\times \Id_{T^2}\).
   Indeed, since \([-1,1]\) is contractible, we can in a first step homotope \(h\) to the constant map \(h_1(x)=1\in T^2\), \(x\in [-1,1]\).
    So we get an equivariant isotopy from \(F\) to \(F_1(x,t)=(f(x),t)\), \((x,t)\in [-1,1]\times T^2\).
    The homeomorphisms of \([-1,1]\) are precisely the strictly monotone continuous self-maps of this interval which leave the boundary invariant.
    As the set of these continuous functions consists of the disjoint union of two convex sets, we get an isotopy from \(F_1\) to \((\pm\Id_{[-1,1]})\times \Id_{T^2}\).

  So we can extend any equivariant homeomorphism \(g\) of \(M_1\) to an equivariant homeomorphism \(G\) of the preimage \(U\) of a collar of \(M/T^2\) in such a way that the restriction of \(G\)
  to \(\partial U=S^1\times T^2\) is of the form \(A\times \Id_{T^2}\), where \(A\in O(2)\subset \Homeo(S^1)\).
  So we can extend \(G\) to all of \(M\) by \(A\times \Id_{T^2}\) on the complement \(D^2\times T^2\) of \(U\).
\end{proof}

\begin{proof}[Proof of Theorem~\ref{sec:gkm-manifolds-2}]
  If \(M_2\) and \(N_2\) are equivariantly homeomorphic, then the same is true for \(M_1\) and \(N_1\).
  Hence, it follows that \(M\) and \(N\) have isomorphic GKM-graphs.

  So we have to prove the other implication.
  To do so, first note that the GKM-graph of \(M\) determines the equivariant homeomorphism types of all fixed point components of codimension-two subtori \(T'\).
  Indeed, the GKM-graphs of these fixed point components consist of those connected subgraphs \(\Gamma'\) of the GKM-graph of \(M\) such that for all edges in \(\Gamma'\) the corresponding isotropy group contains \(T'\).
  Since all fixed point components of \(T'\) are closed, simply connected and $4$-dimensional by assumption, it follows from the classification results in  \cite[Sections 1.12, 5.1, 5.2]{zbMATH03343393}, that the equivariant homeomorphism type of these fixed point components is determined by their GKM-graphs.

  Now, assume that the GKM-graphs of \(M\) and \(N\) are isomorphic.
  Then we have a homeomorphism \(g_0:M^T\rightarrow N^T\).

  Note that \(M_1\) is constructed from \(M^T\) by gluing in several copies of \(S^2\):
  \[M_1=M^T\cup_{f_M}\amalg S^2\]
  and similarly for \(N_1=N^T\cup_{f_N}\amalg S^2\).
  By Lemma~\ref{sec:gkm-manifolds}, we can extend \(g_0\) to a equivariant homeomorphism \(g_1:M_1\rightarrow N_1\).

  By the above remarks, there are finitely many simply connected $4$-dimensional \(T^d\)-manifolds \(S_1,\dots,S_k\) determined by the GKM-graph of \(M\) such that
  \[M_2=M_1\cup_{f_M'} \amalg S_i\]
  and
  \[N_2=N_1\cup_{f_N'} \amalg S_i.\]
  By Lemma~\ref{sec:gkm-manifolds-1}, we can extend \(g_1\) to an equivariant homeomorphism \(g_2:M_2\rightarrow N_2\).
  Hence the claim follows.
\end{proof}




\section{The case where $\dim M^{S^1}=6$ for some circle $S^1\subset T^3$}
\label{sec:dim6}

In this section we prove the following theorem.

\begin{theorem}
  \label{sec:case-were-dim}
    Assume that we are in the third case of Theorem~\ref{thm:amann_kennard}.
  Assume that there is an \(S^1\subset T^{3}\) with  a $6$-dimensional fixed point component \(P\).
   Then \(H^*(M;\mathbb{Z})\cong H^*(\mathbb{C} P^5;\mathbb{Z})\).
 \end{theorem}

For the convenience of the reader, we recall that being in the third case of Theorem~\ref{thm:amann_kennard} refers to the fact that \(M^{10}\) is a closed, simply connected, positively curved manifold with an effective, isometric \(T^3\)-action such that all of the following statements hold: 
   \begin{itemize}
 \item \(H_i(M;\mathbb{Z})=H_i(\mathbb{C}P^5;\mathbb{Z})\) for \(i\leq 3\)
 \item \(H^4(M;\mathbb{Q})\neq 0\)
 \item \(\chi(M)=\chi(\mathbb{C}P^5)\)
 \item There are two involutions in \(T^3\) fixing two different $6$-dimensional submanifolds \(N_1,N_2\subset M\).
 \end{itemize}

 For the proof of the above theorem we need two lemmas.

\begin{lemma}
  \label{sec:proof-4}
  Assume that we are in the third case of Theorem~\ref{thm:amann_kennard}.
  Assume moreover that there is  a $6$-dimensional submanifold \(P\) fixed by a circle such that \(P\) intersects \(N_1\) or \(N_2\) transversely.
  Then \(H^*(M;\mathbb{Z})\cong H^*(\mathbb{C} P^5;\mathbb{Z})\).
\end{lemma}
\begin{proof}
  Without loss of generality, we may assume that  \(\dim P\cap N_1=2\).
  By Theorem~\ref{thm:wilking2}, we have that \(P\hookrightarrow M\) is $4$-connected.
  

Therefore, we have
\[H_i(P;\mathbb{Z})\cong H_i(M;\mathbb{Z}) \cong H_i(\mathbb{C}P^5;\mathbb{Z}) \text{ for }i\leq 3.\]

By Poincar\'e duality, it follows that \(H^4(P;\mathbb{Z})\cong \mathbb{Z}\).
Hence by the universal coefficient theorem we have \(H_4(P;\mathbb{Z})=\mathbb{Z}\).

Since \(P\hookrightarrow M\) is four-connected, it follows that \(H^4(M;\mathbb{Z})\cong \mathbb{Z}\) and \(H_4(M;\mathbb{Z})\) is cyclic and therefore isomorphic to \(\mathbb{Z}\).
By the Euler characteristic formula in Theorem~\ref{thm:amann_kennard}, it follows that \(H^5(M;\mathbb{Z})\cong H_5(M;\mathbb{Z})\) is torsion.
By the universal coefficient theorem, it follows that \(H_5(M;\mathbb{Z})=0\), since \(H^6(M;\mathbb{Z})\cong H_4(M;\mathbb{Z})\cong \mathbb{Z}\) is torsion-free.

Hence, we have shown that \(H_*(M;\mathbb{Z})\) and \(H_*(\mathbb{C} P^5;\mathbb{Z})\) are isomorphic as groups.
It remains to determine the cup product in cohomology.

By Lemma~\ref{lem:wilking3}, there exist \(z\in H^2(M;\mathbb{Z})\), \(x\in H^4(M;\mathbb{Z})\) such that \(zx\) generates \(H^6(M;\mathbb{Z})\).
By Poincar\'e duality, it follows that \(zx^2\) generates \(H^{10}(M;\mathbb{Z})\).

Now there are three cases:

\textbf{Case 1.} \(z^2\neq 0\).
Since \(H^4(M;\mathbb{Z})=\mathbb{Z}\), there is an \(\alpha\in \mathbb{Z}\setminus\{0\}\) such that \(z^2=\alpha x\).
Therefore \(H^*(M;\mathbb{Q})\) is generated as a \(\mathbb{Q}\)-algebra by \(z\).
Hence it is isomorphic to \(H^*(\mathbb{C} P^5;\mathbb{Q})\) as a ring, i.e.
\(M\) is a rational cohomology \(\mathbb{C} P^5\).

By \cite[p. 393]{MR0413144}, it follows that \(P\) is a rational cohomology \(\mathbb{C}P^3\).
We will show that \(P\) is an integral cohomology \(\mathbb{C} P^3\), i.e. \(P\) has the same integral cohomology ring as \(\mathbb{C} P^3\).
When this is shown it follows that \(\alpha=\pm 1\) because \(P\hookrightarrow M\) is four-connected.
Therefore \(H^*(M;\mathbb{Z})\) is generated as a ring by \(z\).
Hence it is isomorphic to \(H^*(\mathbb{C} P^5;\mathbb{Z})\) as rings, i.e.
\(M\) is an integral cohomology \(\mathbb{C} P^5\).

It remains to prove that \(P\) is an integral cohomology \(\mathbb{C} P^3\).
By Theorem~\ref{thm:grove_searle}, we can assume that  \(P\) is not fixed by a $2$-dimensional subtorus of \(T^3\).
Hence, if \(H\) is the kernel of the \(T^3\)-action on \(P\), then \(T^2=T^3/H\) is a $2$-torus which acts effectively and isometrically on \(P\).

Let \(x\in P\) and assume that the isotropy subgroup \(T^2_x\) of \(x\) in \(T^2\) is disconnected.
Denote by \(F\) the \(T^2_x\)-fixed point component in \(P\) containing \(x\).
Then \(F\) is $2$- or $4$-dimensional.
If \(F\) is $2$-dimensional, then \(T_x^2\) is $1$-dimensional by the upper bound for the symmetry rank given by \cite{MR1255926}.
Since \(T_x^2\) is disconnected and acts effectively on \(P\), it follows from an inspection of the \(T_x^2\)-action on the normal space \(N_x(F,P)\) to \(F\), that there is a $4$-dimensional \(F\subset F'\subset P\) such that \(F'\) is fixed by a non-trival subgroup of \(T_x^2\).
So in both cases our claim follows from  Lemma \ref{lem:fang_rong2}.

Therefore, we can assume that the action on \(P\) has only connected isotropy groups and no $4$-dimensional submanifold of \(P\) is fixed by a subgroup of \(T^2\).

Denote by \(p:S\rightarrow P\) the \(S^1\)-bundle whose first Chern class generates \(H^2(P;\mathbb{Z})\).
Since \(P\) is simply connected it follows from an inspection of the Gysin sequence that \(S\) is two-connected.
Again, by the Gysin sequence, to show that \(P\) is an integral cohomology \(\mathbb{C} P^3\), it suffices to show that \(S\) is an integral homology sphere.
By \cite[Corollary 1.3]{0346.57014}, there is a \(T^3\)-action on \(S\) and a homomorphism \(\phi:T^3\rightarrow T^2\) such that:
\begin{itemize}
\item The bundle projection \(p:S\rightarrow P\) is \(\phi\)-equivariant.
\item The kernel of \(\phi\) is connected and $1$-dimensional. Moreover, it acts by multiplication on the fibers of \(p\).
\end{itemize}

Let \(x\in S\) with non-trivial isotropy subgroup \(T_x^3\subset T^3\). Then it follows from these properties of the \(T^3\)-action that the restriction of \(\phi\) induces an isomorphism of \(T^3_x\) and \(T^2_{p(x)}\).
Therefore \(T_x^3\) is connected.
Moreover, the \(T^3_x\)-fixed point component \(F_x\)  in \(S\) containing \(x\) is the restriction of the \(S^1\)-bundle \(S\rightarrow P\) to the \(T_{p(x)}^2\)-fixed point component \(G_{p(x)}\) in \(P\) containing \(p(x)\).

Since the latter component is $0$- or $2$-dimensional and \(P\) is a rational cohomology \(\mathbb{C} P^3\), it follows that \(F_x\) is a $3$-dimensional lens space or a circle.

Hence, it follows from Proposition~\ref{prop:fang_rong3} that \(S\) is a integral homology sphere.

\textbf{Case 2.} \(z^2=0\) and \(N_1\) is fixed by a circle.
Then the integral cohomology ring of \(M\) is isomorphic to that of \(S^2\times \mathbb{H}P^2\).
Since \(P\hookrightarrow M\) and \(N_1\hookrightarrow M\) are four-connected, it follows that \(P\) and \(N_1\) have the integral cohomology of \(S^2\times \mathbb{H} P^1\).

Hence, it follows from Theorem~\ref{thm:kww} that there is a \(T^2\)-fixed point component \(F\) of dimension at least four with cohomology of \(S^2\times \mathbb{C} P^l\).
Here \(T^2\subset T^3\) is the $2$-dimensional torus generated by the two circles which fix \(P\) and \(N_1\), respectively.

If the dimension of \(F\) is six or larger then a contradiction arises from Theorem~\ref{thm:grove_searle}.

If the dimension of \(F\) is four and \(F\) is not fixed by all of \(T^3\) then a contradiction arises from Theorem \ref{thm:hsiang-kleiner}.

If the dimension of \(F\) is four and \(F\) is fixed by \(T^3\) then a contradiction arises from Theorem \ref{thm:grove_searle}.

So this case does not appear.

\textbf{Case 3.} \(z^2=0\) and \(N_1\) is not fixed by a circle.
Then \(N_1\hookrightarrow M\) is three-connected by Theorem~\ref{thm:wilking2}.
Moreover, the \(T^3\)-action on \(N_1\) is almost effective.
Hence, it follows from Theorem \ref{thm:grove_searle} that \(H^*(N_1;\mathbb{Z})\cong H^*(\mathbb{C}P^3;\mathbb{Z})\).
This implies \(z^2|_{N_1}\neq 0\). A contradiction.

So this case does not appear.
\end{proof}

\begin{lemma}
  \label{sec:proof-3}
  Assume that we are in the third case of Theorem~\ref{thm:amann_kennard}.
  Assume moreover that there is  a $6$-dimensional submanifold \(P\) fixed by a circle such that \(P\) intersects \(N_1\) or \(N_2\) in a $4$-dimensional manifold.
   Then \(H^*(M;\mathbb{Z})\cong H^*(\mathbb{C} P^5;\mathbb{Z})\).
\end{lemma}
\begin{proof}
  By Theorem~\ref{thm:wilking2}, the inclusion \(P\hookrightarrow M\) is $4$-connected.
Hence, it follows from Lemma~\ref{lem:fang_rong2} that \(H^*(P;\mathbb{Z})\cong H^*(\mathbb{C}P^3;\mathbb{Z})\).

Hence \(H_4(M;\mathbb{Z})\) is cyclic and therefore isomorphic to \(\mathbb{Z}\).
By Poincare duality \(H^6(M;\mathbb{Z})\) is torsion-free.
Hence by the universal coefficient theorem \(H_5(M)\) is torsion-free and therefore zero by the Euler-characteristic formula in Theorem~\ref{thm:amann_kennard}.

So we have \(H_*(M;\mathbb{Z})\cong H_*(\mathbb{C}P^5;\mathbb{Z})\).
In particular, \(P\hookrightarrow M\) is six-connected.

So if \(x\in H^2(M;\mathbb{Z})\) is a generator, then \(x^2\) and \(x^3\) generate \(H^4(M;\mathbb{Z})\) and \(H^6(M;\mathbb{Z})\), respectively.
By Poincare duality it follows that \(x^5\) generates \(H^{10}(M;\mathbb{Z})\).
Hence, again by Poincare duality, \(x^4\) generates \(H^8(M;\mathbb{Z})\).
Therefore the claim follows.
\end{proof}

  Now we are ready to prove Theorem~\ref{sec:case-were-dim}.

 \begin{proof}[Proof of Theorem \ref{sec:case-were-dim}]
   If \(P\) intersects one of the manifolds \(N_1\) and \(N_2\) transversely, then the claim follows from Lemma~\ref{sec:proof-4}.
   If \(P\) intersects one of the manifolds \(N_1\) and \(N_2\) non-transversely,
   then the claim holds by Lemma~\ref{sec:proof-3}.
 \end{proof}

\section{The case where $\dim M^{S^1}\leq 4$ for all circles $S^1\subset T^3$}
\label{sec:dim4}

In this section we prove the following theorem.

\begin{theorem}
  \label{sec:proof-2}
    Assume that we are in the third case of Theorem~\ref{thm:amann_kennard}.
  Assume that for all \(S^1\subset T^{3}\) we  have \(\dim M^{S^1}\leq 4\).
   Then \(H^*(M;\mathbb{Z})\cong H^*(\mathbb{C} P^5;\mathbb{Z})\).
\end{theorem}

For the convenience of the reader, we recall again that being in the third case of Theorem~\ref{thm:amann_kennard} refers to the fact that \(M^{10}\) is a closed, simply connected, positively curved manifold with an effective, isometric \(T^3\)-action such that all of the following statements hold: 
   \begin{itemize}
 \item \(H_i(M;\mathbb{Z})=H_i(\mathbb{C}P^5;\mathbb{Z})\) for \(i\leq 3\)
 \item \(H^4(M;\mathbb{Q})\neq 0\)
 \item \(\chi(M)=\chi(\mathbb{C}P^5)\)
 \item There are two involutions in \(T^3\) fixing two different $6$-dimensional submanifolds \(N_1,N_2\subset M\).
 \end{itemize}

 We give the proof of the above theorem in a sequence of lemmas.
 
\begin{lemma}
  \label{sec:case-where-dim}
  Assume that we are in the third case of Theorem~\ref{thm:amann_kennard}.
  Assume that for all \(S^1\subset T^{3}\) we  have \(\dim M^{S^1}\leq 4\).
  Then the \(T^3\)-action on \(M\) is GKM$_3$ and the two-skeleton \(M_2\) is equivariantly homeomorphic to the two-skeleton \(\mathbb{C}P^5_2\) of a linear action of \(T^3\) on \(\mathbb{C} P^5\).
\end{lemma}
\begin{proof}
  Since \(\dim M^{S^1}\leq 4\) for all circle subgroups \(S^1\subset T^3\), we have by Lemma~\ref{sec:gkm-manifolds-5} that the \(T^3\)-action on \(M\) is GKM$_3$.
Note that since the action is GKM$_3$ there are only finitely many \(T^3\)-fixed points in \(M\). Moreover, the number of torus fixed points in \(M\) is equal to the Euler characteristic of \(M\) which is six by assumption.
Hence, by Theorem \ref{sec:gkm-manifolds-3}, the GKM graph of the action on \(M\) is combinatorially equivalent to that of a linear torus action on \(\mathbb{C} P^5\) or to that of a linear torus action on \(\amalg_{i=1}^3S^{10}\).

Recall here that the underlying graph of the GKM graph of a linear torus action on \(\mathbb{C} P^5\) is given by a complete graph on six vertices.
Moreover, the underlying graph of the GKM graph of a linear torus action on \(S^{10}\) has two vertices and five edges connecting these vertices.

By assumption the induced \(T^3\)-action on \(N_1\) is almost effective.
Hence, by Theorem~\ref{thm:grove_searle}, \(N_1\) is diffeomorphic to \(\mathbb{C} P^3\).

Let \(H\subset T^3\) be the kernel of the \(T^3\)-action on \(N_1\) and \(x\in N_1^{T^3}\).
Up to automorphisms of the torus \(T^3/H\), the  \(T^3/H\)-representation \(T_xN_1\) is isomorphic to its standard representation on \(\mathbb{C}^3\).  Hence, one can find a submanifold \(N_1'\subset N_1\) of codimension two which is fixed by a circle subgroup of \(T^3\).
By \cite[Theorem 7.5.1]{MR0413144}, \(N_1'\) has the same integral cohomology as \(\mathbb{C} P^2\).

In particular, $\chi(N_1') = 3$  and \(N_1'\) contains three \(T^3\)-fixed points.
Since \(N_1'\) is $4$-dimensional, the GKM graph of \(N_1'\) is a triangle.
Since this is a subgraph of the GKM graph of \(M\), it follows that the GKM graph of \(M\) must be the same as the GKM graph of a linear torus action on \(\mathbb{C} P^5\).

By Theorem~\ref{sec:gkm-manifolds-2} there is an equivariant homeomorphism \(M_2\rightarrow \mathbb{C} P^5_2\).
\end{proof}

\begin{lemma}
  \label{sec:proof-5}
  Assume that we are in the third case of Theorem~\ref{thm:amann_kennard}.
  Assume that for all \(S^1\subset T^{3}\) we  have \(\dim M^{S^1}\leq 4\).
  Then \(H^*(M;\mathbb{Q})\cong H^*(\mathbb{C} P^5;\mathbb{Q})\).
\end{lemma}
\begin{proof}
We first claim that \(b_5(M)=0\).

To prove this claim, first note that by the third case of Theorem~\ref{thm:amann_kennard} and Poincar\'e duality, we have
\begin{equation}
  \label{eq:1}
b_0(M)=b_2(M)=1,\quad b_1(M)=b_3(M)=0,\quad b_4(M)=k+1,\quad b_5(M)=2k
\end{equation}
 for some non-negative \(k\in \mathbb{Z}\).

 By the proof of Lemma~\ref{sec:case-where-dim}, we know \(N_1=\mathbb{C} P^3\).
 Moreover, \(N_1\hookrightarrow M\) is three-connected by the Connectedness Lemma (Theorem \ref{thm:wilking2}).
 Therefore, we have that the map from the upper left corner to the lower right corner of the following diagram \[\xymatrix{H_2((N_1)_2;\mathbb{Q})\ar[r]\ar[d]& H_2(M_2;\mathbb{Q})\ar[d]\\ H_2(N_1;\mathbb{Q})\ar[r]& H_2(M;\mathbb{Q})}\] is surjective.
 
Hence, by Lemma~\ref{sec:case-where-dim}, the maps \(H_*(M_2;\mathbb{Q})\rightarrow H_*(M;\mathbb{Q})\) and \(H_*(\mathbb{C} P^5_2;\mathbb{Q})\rightarrow H_*(\mathbb{C}P^5;\mathbb{Q})\), \(*\leq 3\), have isomorphic kernels and cokernels.
In particular,
\begin{equation}
  \label{eq:2}
  H_*(M,M_2;\mathbb{Q})\cong H_*(\C P^5,\mathbb{C} P^5_2;\mathbb{Q}) \text{ for } *\leq 3.
\end{equation}

In particular,
\begin{equation}
\label{eq:3}
  H_0(M,M_2;\mathbb{Q})=H_1(M,M_2;\mathbb{Q})=H_2(M,M_2;\mathbb{Q})=0.\end{equation}

Let \(X=M- U\), with \(U\) an open \(T^3\)-invariant neighborhood  of \(M_2\).
Then we can assume the following:
\begin{itemize}
\item \(X\) is a manifold with boundary.
\item \(U\) is homotopy equivalent to \(M_2\).
\item \(X\hookrightarrow M\) is five-connected, because \(\dim M_2=4\).
\item \(T^3\) acts almost freely on \(X\), because all points in \(M\) with positive dimensional isotropy group are contained in \(M_2\).
\end{itemize}

By Poincare duality and excision we get \[H^*(X;\mathbb{Q})\cong H_{10-*}(X,\partial U;\mathbb{Q})\cong H_{10-*}(M,M_2;\mathbb{Q}).\]
Combining this with the fact that \(X\hookrightarrow M\) is five-connected and using Equations (\ref{eq:1}), (\ref{eq:2}) and (\ref{eq:3}) we get
\[b_0(X)=b_2(X)=1,\; b_1(X)=b_3(X)=0,\;\] \[b_4(X)=k+1,\; b_7(X)=c,\; b_i(X)=0, \text{ for }i\geq 8.\] Here \(c=\dim H_3(\C P^5,\mathbb{C} P^5_2;\mathbb{Q})\) is a universal constant which is independent of \(M\).

Let \(T^3=S^1_1\times S_2^2\times S_3^1\). Then from the spectral sequence from Lemma \ref{sec:equiv-cohom-local-1} for \(X \rightarrow X/S^1_1\), we get
\[b_0(X/S^1_1)=1,\; b_2(X/S^1_1)=2,\; b_1(X/S^1_1)=b_3(X/S^1_1)=0,\;\] \[b_4(X/S^1_1)=k+3,\; b_6(X/S^1_1)=c,\; b_i(X/S^1_1)=0,\text{ for } i\geq 7.\]

To give more detail: The \(E_2\)-page of this spectral sequence has only two isomorphic consecutive non-zero lines. Hence all differentials \(d_r:E_r^{*,*}\rightarrow E_r^{*,*}\), \(r\geq 3\) are zero.
Hence we have \(b_i(X)=b_i(X/S^1_1)-\dim \image d_{2,i-2} + \dim \ker d_{2,i-1}\), where \(d_{2,i}:E^{1,i}_2\rightarrow E_2^{0,i+2}\) is the differential of the \(E_2\)-page.
Since \(X/S^1_1\) has finite cohomological dimension the results on the Betti numbers follow.

Next, from the spectral sequence from Lemma \ref{sec:equiv-cohom-local-1} for \(X/S^1_1 \rightarrow X/T^2\), \(T^2=S^1_1\times S^1_2\), we get
\[b_0(X/T^2)=1,\; b_2(X/T^2)=3,\; b_1(X/T^2)=b_3(X/T^2)=0,\;\] \[b_4(X/T^2)=k+6,\; b_5(X/T^2)=c,\; b_i(X/T^2)=0, \text{ for } i\geq 6.\]

Finally, from the same spectral sequence for \(X/T^2 \rightarrow X/T^3\), we get
\[b_0(X/T^3)=1,\; b_2(X/T^3)=4,\; b_1(X/T^3)=b_3(X/T^3)=0,\;\] \[b_4(X/T^3)=k+10=c,\; b_i(X/T^3)=0, \text{ for } i\geq 5.\]
In particular, \(k\) is independent of \(M\) and must therefore be zero because all our arguments also apply in the case \(M=\mathbb{C} P^5\).
Since \(b_5(M)=2k\), the claim follows.

Now, Theorem~\ref{sec:gkm-manifolds-4} together with Lemma~\ref{sec:case-where-dim} and Theorem \ref{sec:gkm-manifolds-2} implies \(H^*(M;\mathbb{Q})\cong H^*(\mathbb{C} P^5;\mathbb{Q})\).
\end{proof}

\begin{lemma}
  \label{sec:proof-7}
  Assume that we are in the third case of Theorem~\ref{thm:amann_kennard}.
  Assume that for all \(S^1\subset T^{3}\) we  have \(\dim M^{S^1}\leq 4\).
  Then \(H^*(M;\mathbb{Z})/\text{torsion}\cong H^*(\mathbb{C} P^5;\mathbb{Z})\) as rings.
\end{lemma}
\begin{proof}
By Lemma~\ref{sec:proof-5} and the universal coefficient theorem, we have \[H_*(M;\mathbb{Z})/\text{torsion}\cong H_*(\mathbb{C}P^5,\mathbb{Z})\] as groups.
Moreover \(N_1\hookrightarrow M\) is three-connected and \(N_1\) is diffeomorphic to \(\mathbb{C} P^3\) by Theorem~\ref{thm:grove_searle}.
Hence, it follows by an application of Poincare duality as in the proof of Lemma~\ref{sec:proof-3}, that \(H^*(M;\mathbb{Z})/\text{torsion}\cong H^*(\mathbb{C}P^5;\mathbb{Z})\) as rings.
\end{proof}

\begin{lemma}
  \label{sec:proof-6}
  Assume that we are in the third case of Theorem~\ref{thm:amann_kennard}.
  Assume that for all \(S^1\subset T^{3}\) we  have \(\dim M^{S^1}\leq 4\).
   Let \(G\subset T^3\) be a finite non-trivial subgroup. Then \(H^*(M^G;\mathbb{Z})\) is a free \(\mathbb{Z}\)-module concentrated in even degrees.
\end{lemma}
\begin{proof}
It suffices to consider the case where \(M^G\) is connected. If \(M^G\) is fixed by a circle group, then by assumption it has dimension at most four and the claim holds by Theorem~\ref{thm:hsiang-kleiner} or Gauss-Bonnet.

If \(M^G\) is not fixed by a circle group, then, by the upper bound on the symmetry rank of a positively curved manifold from \cite{MR1255926}, it has dimension six or eight and admits an effective isometric \(T^3\)-action. Hence the claim follows from Theorems \ref{thm:grove_searle} and \ref{thm:fang_rong1}, respectively.
\end{proof}

\begin{lemma}
  \label{sec:proof-1}
  Let \(M\) be a \(10\)-dimensional manifold with \(T^2\)-action such that:
  \begin{enumerate}
  \item \(H^*(M;\mathbb{Z})/\text{torsion}\cong H^*(\mathbb{C} P^5;\mathbb{Z})\) as rings
  \item the torsion in \(H^*(M;\mathbb{Z})\) is contained in degrees five and six,
  \item for all \(p\in\mathbb{Z}\) prime and \(\mathbb{Z}_p\subset S^1\subset T^2\), \(H^*_{S^1}(M^{\mathbb{Z}_p};\mathbb{Z})\)
    has no \(H^*(BS^1;\mathbb{Z})\)-torsion,
      \item for all \(p\in\mathbb{Z}\) prime and \(\mathbb{Z}_p\times \mathbb{Z}_p\subset T^2\), \(H^*_{T^2}(M^{\mathbb{Z}_p\times \mathbb{Z}_p};\mathbb{Z})\)
 has no \(H^*(BT^2;\mathbb{Z})\)-torsion,
\end{enumerate}
Then \(H^*(M;\mathbb{Z})\) is torsion-free.
In particular,  \(H^*(M;\mathbb{Z})\cong H^*(\mathbb{C} P^5;\mathbb{Z})\) as rings.
\end{lemma}
\begin{proof}
  Let \(p\in \mathbb{Z}\) be a prime.
  Assume that there is \(p\)-torsion in \(H^*(M;\mathbb{Z})\).
  Then by Poincare duality and the universal coefficient theorem there is \(p\)-torsion in both degree \(5\) and \(6\).
  
  We show the following:

  \textbf{Claim:} There is \(S^1\subset T^2\) such that \(H^*_{S^1}(M;\mathbb{Z})\) contains a \(p\)-torsion element which is not \(t\)-torsion, \(t\in H^2(BS^1;\mathbb{Z})\) a generator.

  When this is shown, a contradiction arises from localization (Theorem~\ref{sec:equiv-cohom-local-2} and Lemma~\ref{sec:equiv-cohom-local-3})  since then
  \[S^{-1}_pH^*_{S^1}(M)\cong S^{-1}_pH^*_{S^1}(M^{\mathbb{Z}_p})\]
  contains \(p\)-torsion, where \(S_p=\{at^k;\; a\in \mathbb{Z}, p\not| a, k\geq 0\}\).
  But by assumption it does not contain torsion.

  So we prove the claim:
  Consider the Serre spectral sequence  \((E_*,d_*)\) for \(M\rightarrow M_T\rightarrow BT\).

  The only differential which might be non-zero in this spectral sequence is \(d_2:E^{6,*}_2\rightarrow E^{5,*+2}_2\).
  Note that \(H^{\text{odd}}_{T^2}(M;\mathbb{Z})\cong E_\infty^{5,*}\).
  Hence, by localization (Theorem~\ref{sec:equiv-cohom-local-2} and Lemma~\ref{sec:equiv-cohom-local-3}) and assumption (4) for each \(p\)-torsion element \(a\in H^5(M)=E_2^{5,0}\), there is a \(s\in S_p=\{\prod_i s_i;\;s_i\in H^2(BT), p\not| s_i\}\) such that \(sa\) is in the image of \(d_2\).
  Let \(b\) be a preimage. Then we may assume that \(b\) is \(p\)-torsion.
  Moreover, we may assume that \(s\) and \(b\) do not have common divisors in \(H^*(BT;\mathbb{Z})\).
  Let \(t_1\in H^2(BT;\mathbb{Z})\) be a primitive element dividing \(s\) and \(S^1=\ker t_1\subset T^2\).

  Consider the Serre spectral sequence \((F_*,d'_*)\) for \(M\rightarrow M_{S^1}\rightarrow BS^1\).
  Then we have a map \(f:E_*\rightarrow F_*\).
  Moreover we have \(f(b)\neq 0\) and \(d_2'(f(b))=0\).
  Hence, \(c=f(b)\) survives the spectral sequence and we have \(p\)-torsion elements in \(F_\infty^{6,*}\) which are not \(t\)-torsion.

  We have a filtration
  \[H^*_{S^1}(M;\mathbb{Z})=K_{10}\supset K_9\supset K_8 \supset \dots K_0=H^*(BS^1;\mathbb{Z}),\]
  by graded \(H^*(BS^1;\mathbb{Z})\)-modules
  such that there are exact sequences
  \[0\leftarrow F_\infty^{i,*}\leftarrow K_i \leftarrow K_{i-1}\leftarrow 0\]
  for \(i=1,\dots, 10\).

  For \(i\leq 4\), \(F_\infty^{i,*}\) is concentrated in even degrees.
  Hence, \(K_4\) is concentrated in even degrees.
  Moreover, \(K_4\) is generated as a module over \(H^*(BS^1;\mathbb{Z})\) by \(1,x,x^2\) where \(x\in H^2_{S^1}(M;\mathbb{Z})\) is a lift of a generator of \(H^2(M;\mathbb{Z})\).

  Note also that \(F_\infty^{5,*}\) is concentrated in odd degrees.
  Hence, it follows that the even-degree part of \(K_5\) is equal to \(K_4\).
  
  We claim that \(K_6\) contains a \(p\)-torsion element which is not a \(t\)-torsion element.

  Let \(\tilde{c}\in K_6\) be a lift of the \(p\)-torsion element \(c\in F_\infty^{6,*}\) and let \(\gamma\) be half of its degree.
  Then \(\tilde{c}\) is not \(t\)-torsion since \(c\) is not \(t\)-torsion.

  We can assume that \(pc=0\). Then there are \(a_0,a_1,a_2\in \mathbb{Z}\) such that
  \[p\tilde{c}=\sum_{i=0}^2 a_i t^{\gamma-i}x^i.\]
  By integration over the fiber (see Section~\ref{sec:prelim_equi_cohom}) we get the following equation in \(H^*(BS^1;\mathbb{Z})\):
  \[p\tilde{c}x^3[M]=\sum_{i=0}^2 a_i t^{\gamma-i}x^{i+3}[M]=a_2t^{\gamma-2}x^5[M]=a_2t^{\gamma-2}.\]
  Here, for the last equality we have used that \(H^*(M;\mathbb{Z})/\text{torsion}\cong H^*(\mathbb{C} P^5;\mathbb{Z})\).
  So \(p\) divides \(a_2\).
  Similarly we have
  \[(p\tilde{c}x^4-a_2t^{\gamma-2}x^6)[M]=\sum_{i=0}^1 a_i t^{\gamma-i}x^{i+4}[M]=a_1t^{\gamma-1}x^5[M]=a_1t^{\gamma-1}.\]
  Hence, \(p\) also divides \(a_1\). Similarly, one sees that \(p\) also divides \(a_0\).

  Now let \(\tilde{c}'=\tilde{c}-\sum_{i=0}^2\frac{a_i}{p} t^{\gamma-i} x^i\).
  Then \(\tilde{c}'\) is a lift of \(c\) which is \(p\)-torsion but not \(t\)-torsion.
  Hence our claim follows.
\end{proof}

Now we are ready to prove Theorem~\ref{sec:proof-2}.

\begin{proof}[Proof of Theorem \ref{sec:proof-2}]
    By Lemma \ref{sec:proof-7}, the Assumption (1) of Lemma~\ref{sec:proof-1} is satisfied.
    
    Assumption (2) of Lemma~\ref{sec:proof-1} follows from Poincare duality and the universal coefficient theorem because \(H_i(M;\mathbb{Z})\) is torsion-free for \(0\leq i\leq 3\).

    Finally Assumptions (3) and (4) of Lemma~\ref{sec:proof-1} are satisfied by  Lemmas \ref{sec:proof-6} and \ref{sec:equiv-cohom-local}. Therefore the claim follows from Lemma~\ref{sec:proof-1}.
\end{proof}

\section{The proof of the Main Theorem}
\label{sec:complete}

In this final section we prove our main result.

\begin{proof}[Proof of the Main Theorem]
  Assume that we are in the situation of the Main Theorem, i.e. \(M^{10}\) is a closed, simply connected, positively curved manifold with an isometric effective action of a torus \(T^3\).

  If we are not in the third case of Theorem~\ref{thm:amann_kennard}, then \(M\) has the integral cohomology of \(S^{10}\) or \(\mathbb{C} P^5\).

  So we can assume that we are in the third case of Theorem~\ref{thm:amann_kennard} and not in one of the other cases.
  So, in particular, the codimension of the fixed point sets of every \(S^1\subset T^3\) is at least four.
  
  If there is a \(S^1\subset T^3\) with a $6$-dimensional fixed point component, then it follows from Theorem~\ref{sec:case-were-dim} that \(M\) has the integral cohomology of \(\mathbb{C} P^5\).

  If there is no such \(S^1\subset T^3\), then it follows from Theorem~\ref{sec:proof-2} that the same holds in this case.
  
  It is known that if \(\pi_1(M)=0\) and $M$ has the integral cohomology ring of $S^{10}$ or $\C P^5$, then \(M\) is homotopy equivalent to one of these two spaces.
  For the sake of completeness we give an argument for this fact here.

  First assume that \(M\) has the integral cohomology of \(S^{10}\).
  Since there is a CW-complex-structure on \(M\) with exactly one \(10\)-cell, there is a bijection
  \[[M;S^{10}]\rightarrow H^{10}(M;\mathbb{Z})\quad\quad [f]\mapsto f^*u,\]
  where \(u\) is a generator of \(H^{10}(S^{10};\mathbb{Z})\) and \([M;S^{10}]\) denotes the set of homotopy classes of maps \(M\rightarrow S^{10}\).
  Hence, in this case the claim follows from Whitehead's theorem applied to a map \(f:M\rightarrow S^{10}\) corresponding to a generator of \(H^{10}(M;\mathbb{Z})\).

  Next assume that \(M\) has the integral cohomology of \(\mathbb{C} P^5\).
  Note that \(\mathbb{C}P^\infty\) is an Eilenberg--MacLane space \(K(\mathbb{Z},2)\).
  Moreover, the $11$-skeleton of the CW-complex \(\mathbb{C} P^{\infty}\) is given by
  \(\mathbb{C} P^5\subset \mathbb{C} P^\infty\).
  Since \(M\) is $10$-dimensional, we therefore have \([M;\mathbb{C} P^5]\cong [M;\mathbb{C} P^\infty]\).

  Hence, it follows that there is a bijection
  \[[M;\mathbb{C}P^{5}]\rightarrow H^{2}(M;\mathbb{Z})\quad\quad [f]\mapsto f^*u,\]
  where \(u\) is a generator of \(H^2(\mathbb{C} P^5;\mathbb{Z})\).
    Therefore, in this case the claim follows from Whitehead's theorem applied to a map \(f:M\rightarrow \mathbb{C} P^5\) corresponding to a generator of \(H^{2}(M;\mathbb{Z})\).
\end{proof}


\bibliography{refs}{}

\begin{thebibliography}{KWW22}

\bibitem[AK20]{MR4130255}
Manuel Amann and Lee Kennard.
\newblock Positive curvature and symmetry in small dimensions.
\newblock {\em Commun. Contemp. Math.}, 22(6):1950053, 57, 2020.

\bibitem[Ber61]{MR133083}
M.~Berger.
\newblock Les vari\'et\'es riemanniennes homog\`enes normales simplement
  connexes \`a{} courbure strictement positive.
\newblock {\em Ann. Scuola Norm. Sup. Pisa Cl. Sci. (3)}, 15:179--246, 1961.

\bibitem[BH58]{zbMATH03159124}
A.~Borel and F.~Hirzebruch.
\newblock Characteristic classes and homogeneous spaces. {I}.
\newblock {\em Amer. J. Math.}, 80:458--538, 1958.

\bibitem[Bre72]{MR0413144}
Glen~E. Bredon.
\newblock {\em Introduction to compact transformation groups}.
\newblock Pure and Applied Mathematics, Vol. 46. Academic Press, New
  York-London, 1972.

\bibitem[DW04]{zbMATH02133491}
Anand Dessai and Burkhard Wilking.
\newblock Torus actions on homotopy complex projective spaces.
\newblock {\em Math. Z.}, 247(3):505--511, 2004.

\bibitem[FR05]{MR2139252}
Fuquan Fang and Xiaochun Rong.
\newblock Homeomorphism classification of positively curved manifolds with
  almost maximal symmetry rank.
\newblock {\em Math. Ann.}, 332(1):81--101, 2005.

\bibitem[GKM98]{MR1489894}
Mark Goresky, Robert Kottwitz, and Robert MacPherson.
\newblock Equivariant cohomology, {K}oszul duality, and the localization
  theorem.
\newblock {\em Invent. Math.}, 131(1):25--83, 1998.

\bibitem[GS94]{MR1255926}
Karsten Grove and Catherine Searle.
\newblock Positively curved manifolds with maximal symmetry-rank.
\newblock {\em J. Pure Appl. Algebra}, 91(1-3):137--142, 1994.

\bibitem[GW15]{MR3456711}
Oliver Goertsches and Michael Wiemeler.
\newblock Positively curved {GKM}-manifolds.
\newblock {\em Int. Math. Res. Not. IMRN}, 2015(22):12015--12041, 2015.

\bibitem[HK89]{MR992332}
Wu-Yi Hsiang and Bruce Kleiner.
\newblock On the topology of positively curved {$4$}-manifolds with symmetry.
\newblock {\em J. Differential Geom.}, 29(3):615--621, 1989.

\bibitem[HY76]{0346.57014}
A.~Hattori and T.~Yoshida.
\newblock {Lifting compact group actions in fiber bundles.}
\newblock {\em Jap. J. Math., new. Ser.}, 2:13--25, 1976.

\bibitem[Kaw91]{zbMATH00051915}
Katsuo Kawakubo.
\newblock {\em The theory of transformation groups}.
\newblock Oxford etc.: Oxford University Press, 1991.

\bibitem[Ken13]{zbMATH06152271}
Lee Kennard.
\newblock On the {Hopf} conjecture with symmetry.
\newblock {\em Geom. Topol.}, 17(1):563--593, 2013.

\bibitem[KWW21]{KWW}
Lee Kennard, Michael Wiemeler, and Burkhard Wilking.
\newblock Splitting of torus representations and applications in the {Grove}
  symmetry program.
\newblock Preprint, arXiv:2106.14723, 2021.

\bibitem[KWW22]{KWW2}
Lee Kennard, Michael Wiemeler, and Burkhard Wilking.
\newblock Positive curvature, torus symmetry and matroids.
\newblock Preprint, arXiv:2212.08152, 2022.

\bibitem[May89]{zbMATH04100211}
Karl~Heinz Mayer.
\newblock G-invariante {Morse}-{Funktionen}. ({G}-invariant {Morse} functions).
\newblock {\em Manuscr. Math.}, 63(1):99--114, 1989.

\bibitem[McC01]{MR1793722}
John McCleary.
\newblock {\em A user's guide to spectral sequences}, volume~58 of {\em
  Cambridge Studies in Advanced Mathematics}.
\newblock Cambridge University Press, Cambridge, second edition, 2001.

\bibitem[Nie22]{Nienhaus}
Jan Nienhaus.
\newblock An improved four-periodicity theorem and a conjecture of {H}opf with
  symmetry.
\newblock Preprint, arXiv:2211.13151, 2022.

\bibitem[Nov10]{novikov10:_homot}
S.~P. Novikov.
\newblock Homotopically equivalent smooth manifolds.
\newblock In S.~P. Novikov and I.~A. Taimanov, editors, {\em Topological
  library. {P}art 2: {C}haracteristic classes and smooth structures on
  manifolds}, volume~44 of {\em Series on Knots and Everything}, pages 49--183.
  World Scientific Publishing Co. Pte. Ltd., Hackensack, NJ, 2010.
\newblock Translated from the Russian original by V. O. Manturov.

\bibitem[OR70]{zbMATH03343393}
P.~Orlik and F.~Raymond.
\newblock Actions of the torus on 4-manifolds. {I}.
\newblock {\em Trans. Am. Math. Soc.}, 152:531--559, 1970.

\bibitem[tD87]{zbMATH03988265}
Tammo tom Dieck.
\newblock {\em Transformation groups}, volume~8 of {\em De Gruyter Stud. Math.}
\newblock De Gruyter, Berlin, 1987.

\bibitem[Wil03]{MR2051400}
Burkhard Wilking.
\newblock Torus actions on manifolds of positive sectional curvature.
\newblock {\em Acta Math.}, 191(2):259--297, 2003.

\end{thebibliography}
\bibliographystyle{alpha}
\end{document}